\documentclass[12pt,reqno,letter]{amsart}

\usepackage[all]{xy}

\def\theoremnumberstyle{plain}
\def\proofname{Proof}
\usepackage{tom}

\newcommand{\tom}[1]{}

\newcommand{\pp}{{\mathbb P}}

\SelectTips{cm}{12}

\begin{document}

   \parindent0cm

   \title[Triple Point]{Triple-Point Defective Regular Surfaces}
   
   \author{Luca Chiantini}
      \address{Universit\'a degli Studi di Siena\\
     Dipartimento di Scienze Matematiche e Informatiche\\
     Pian dei Mantellini, 44
     I -- 53100 Siena
     }
   \email{chiantini@unisi.it}
  \urladdr{http://www.dsmi.unisi.it/newsito/docente.php?id=4}
   
   \author{Thomas Markwig}
   \address{Universit\"at Kaiserslautern\\
     Fachbereich Mathematik\\
     Erwin-Schr\"odinger-Stra\ss e\\
     D -- 67663 Kaiserslautern
     }
   \email{keilen@mathematik.uni-kl.de}
   \urladdr{http://www.mathematik.uni-kl.de/\textasciitilde keilen}
   \thanks{The second author was supported by the EAGER node of Torino,
     and by the Institute for Mathematics and its Applications (IMA),
     University of Minnesota.}


   \date{30 November, 2006.}


   \begin{abstract}
     In this paper we study the linear series $|L-3p|$ of hyperplane
     sections with a triple point $p$ of a surface $S$ embedded via a very
     ample line bundle $L$ for a \emph{general} point $p$. If this linear
     series does not have the expected dimension we call $(S,L)$
     \emph{triple-point defective}. We show that on a triple-point defective
     \emph{regular} surface through a general point every hyperplane section
     has either a triple component or the surface is rationally ruled
     and the hyperplane section contains twice a fibre of the ruling. 
   \end{abstract}

   \maketitle

   \emph{The results of this paper have been generalised in the paper
     \emph{Triple point defective surfaces} (arXiv:0911.1222) by the same authors. Large
     parts of the latter paper coincide with this paper and the reader
     should rather refer to that paper than to this one.}

   \section{Introduction}\label{sec:tpd}

   Throughout this note, $S$ will be a smooth projective surface,
   $K=K_S$ will denote the canonical class and $L$ will be a
   divisor class on $S$ such that $L$ and $L-K$ are both \emph{very ample}.

   The classical \emph{interpolation problem} for the pair $(S,L)$ is devoted
   to the study of the varieties:
   \begin{displaymath}
     V^{gen}_{m_1,\dots,m_n}=\big\{C\in |L|\;\big|\; p_1,\ldots,p_n\in
     S\mbox{ general},\;\mult_{p_i}(C)
     \geq m_i\big\}.
   \end{displaymath}

   In a more precise formulation, we start from the incidence variety:
   \begin{displaymath}
     \kl_{m_1,\dots,m_n}=\{(C,(p_1,\ldots,p_n))\in|L|\times S^n\;|\;\mult_{p_i}(C)\geq m_i\}
   \end{displaymath}
   together with the canonical projections:
   \begin{equation}\label{eq:alphabeta}
     \xymatrix{
       \kl_{m_1,\dots,m_n}\ar[r]^\alpha\ar[d]_\beta & S^n\\
       |L|=\pp(H^0(L)^*)
     }
   \end{equation}
   As for the map $\alpha$, the fibre  over a fixed point $(p_1,\dots,p_n)\in S^n$
   is just the linear series $|L-m_1p_1-\dots-m_np_n|$ of effective divisors in
   $|L|$ having a point of multiplicity at least $m_i$ at $p_i$.
   These fibres being irreducible, we deduce that if $\alpha$ is
   \emph{dominant} then $\kl_{m_1,\dots,m_n}$ has a unique
   irreducible component, say $\kl_{m_1,\dots,m_n}^{gen}$, which dominates $S$.
   The closure of its image
   \begin{displaymath}\label{VVV}
     V_{m_1,\dots,m_n}:=V_{m_1,\dots,m_n}(S,L):=\overline{\beta(\kl_{m_1,\dots,m_n}^{gen})}
   \end{displaymath}
   under $\beta$ is an irreducible closed subvariety of
   $|L|$, a \emph{Severi variety} of $(S,L)$.

   Imposing a point of multiplicity $m_i$ corresponds to killing 
   $\binom{m_i+1}2$ partial derivatives, so that
   \begin{displaymath}
     \dim|L-m_1p_1-\dots-m_np_n|\geq \max\left\{-1,\dim|L|-\sum_{i=1}^n\binom{m_i+1}2\right\},
   \end{displaymath}
   and we expect that the previous inequality is in fact an equality, for the choice
   of general points $p_1,\dots,p_n\in S$.

   When this is not the case, then the surface is called \emph{defective}
   and is endowed with some special structure.
   \smallskip

   The case when $m_i=2$ for all $i$ has been classically considered (and solved)
   by Terracini, who classified in \cite{Ter22} double--point defective surfaces. 
   In any event, the first example of such a defective surface which is smooth is the
   Veronese surface, for which $n=2$. 

   It is indeed classical that imposing multiplicity two at a general point 
   to a very ample line bundle $|L|$ always 
   yields three independent conditions, so that $\dim|L-2p|=\dim |L|-3$ and
   the corresponding Severi variety has codimension $1$ in $|L|$.

   Furthermore, when $S$ is double-point defective, then any general curve 
   $C\in|L-2p_1-\dots-2p_n|$ has a double component passing through each point $p_i$.

   When the multiplicities grow, the situation becomes completely different.
   Even in the case $S=\pp^2$, the situation is not understood
   and there are several, still unproved conjecture on the structure
   of defective embeddings (see \cite{Cil01} for an introductory survey).

   When $S$ is a more complicated surface, it turns out that even imposing
   just one point of multiplicity $3$, one may expect to obtain
   a defective behaviour.

   \begin{example}\label{ex:hirzebruch}
     Let $S=\F_e\stackrel{\pi}{\longrightarrow}\pp^1$ be a
     Hirzebruch surface, $e\geq 0$. We denote by
     $F$ a fibre of $\pi$ and by $C_0$ the section of $\pi$ of
     minimal self intersection $C_0^2=-e$ -- both of which are smooth
     rational curves.
     The general element $C_1$ in the linear system
     $|C_0+eF|$ will be a section of $\pi$ which does not
     meet $C_0$ (see e.g.\ \cite{FuP00}, Theorem 2.5).

     Consider now the divisor $L=2\cdot F+C_1=(2+e)\cdot
     F+C_0$. Then for a
     general $p\in S$ there are curves $C_p\in|C_1-p|$ and there is a
     unique curve $F_p\in |F-p|$, in particular $p\in F_p\cap C_p$. For each
     choice of $C_p$ we have
     \begin{displaymath}
       2F_p+C_p \in|L-3p|.
     \end{displaymath}
     Since $F.L=1=F.(L-F)$ we see that every curve in $|L-3p|$ must
     contain $F_p$ as a double component, i.e.
     \begin{displaymath}
       |L-3p|=2F_p+|C_1-p|.
     \end{displaymath}
     Moreover, since $p\in S$ is general we have (see \cite{FuP00}, Lemma 2.10)
     \begin{displaymath}
       \dim|C_1-p|=\dim|C_1|-1=h^0\big(\pp^1,\ko_{\pp^1}\big)+
       h^0\big(\pp^1,\ko_{\pp^1}(e)\big)-2=e
     \end{displaymath}
     and, using the notation from above,
     \begin{displaymath}
       \dim(V_3)\geq \dim|C_1-p|+2=e+2.
     \end{displaymath}
     However,
     \begin{displaymath}
       \dim|L|= h^0\big(\pp^1,\ko_{\pp^1}(2)\big)
       +h^0\big(\pp^1,\ko_{\pp^1}(2+e)\big)-1
       =e+5,
     \end{displaymath}
     and thus
     \begin{displaymath}
       \expdim(V_3)=\dim|L|-4=e+1<e+2=\dim(V).
     \end{displaymath}
     We say, $(\F_e,L)$ is \emph{triple-point defective}, see
     Definition \ref{def:tpd}.

     Note, moreover, that
     \begin{displaymath}
       (L-K)^2=(4F+3C_0)^2=24>16.
     \end{displaymath}
     \hfill$\Box$
   \end{example}

   It is interesting to observe that, even though, in the previous example,
   the general element of $|L-3p|$ is non reduced, still the map $\beta$
   of Diagram \eqref{eq:alphabeta} has finite general fibers, since
   the general element of $|L-3p|$ has no triple components.
   \smallskip

   The aim of this note is to investigate the structure of pairs $(S,L)$ 
   for which the linear system $|L-3p|$ for $p\in S$ general has dimension bigger that the expected value
   $\dim|L|-6$, or equivalently, the variety $\kl^{gen}_3$, defined as
   in Diagram \eqref{VVV}, has dimension bigger than  
   $\dim|L|-4$.

   \begin{definition}\label{def:tpd}
     We say that the pair $(S,L)$ is \emph{triple-point defective}  or, in classical
     notation, that \emph{$(S,L)$ satisfies one Laplace equation} if
     \begin{displaymath}
       \dim|L-3p|>\max\{-1,\dim|L|-6\}=\expdim|L-3p|
     \end{displaymath}
     for $p\in S$ general.
   \end{definition}

   \begin{remark}\label{L3}
     Going back to Diagram \eqref{VVV}, one sees that $(S,L)$ is
     triple-point defective if and only if either: 
     \begin{itemize}
     \item $\dim|L|\leq 5$ and the projection
       $\alpha:\kl_3\rightarrow S$ dominates, or \smallskip
     \item $\dim|L|>5$ and the general fibre of the map $\alpha$ has
       dimension at least $\dim|L|-5$.
     \end{itemize}
     In particular, $(S,L)$ is triple-point defective if and only if the map $\alpha$ is
     \emph{dominant} and
     \begin{displaymath}
       \dim(\kl_3^{gen})>\dim|L|-4.
     \end{displaymath}
   \end{remark}

   The particular case in which the general fiber of the map $\beta$ in 
   Diagram \eqref{VVV} is positive-dimensional,
   (i.e.  the general member of $V_3$ contains a triple
   component through $p$) has been investigated in \cite{Cas22},
   \cite{FrI01}, and \cite{BoC05}. We will recall the classification of such surfaces
   in Theorem \ref{thm:notfinite} below.

   Even when $\beta$ is generically finite, one of the major subjects
   in algebraic interpolation theory, namely Segre's conjecture on
   defective linear systems \emph{in the plane},  
   says in our situation that, when $(S,L)$ is triple-point defective,
   then the general element of $|L-3p|$ must be non-reduced, with a double component
   through $p$ (exactly as in the case of Hirzebruch surfaces).

   We are able to show, under some assumptions, that this part of Segre's conjecture 
   holds, even in the more general setting of \emph{regular} surfaces. 

   Indeed our main result is:

   \begin{theorem}\label{thm:aim1}
     Let $S$ be a \emph{regular} surface, and
     suppose that with the notation
     in \eqref{eq:alphabeta} $\alpha$ is dominant. 
     Let $L$ be a very ample line bundle on $S$, such that $L-K$
     is also very ample. Assume $(L-K)^2>16$ and $(S,L)$ is triple-point defective.

     Then $S$ is rationally ruled in the embedding defined by $L$.
     Moreover a general curve $C\in|L-3p|$ contains the fibre of the
     ruling through $p$ 
     as fixed component with multiplicity at least two. 
   \end{theorem}

   \begin{remark}
     In a forthcoming paper \cite{CM07a} we classify all triple-point
     defective linear systems $L$ on ruled surfaces satisfying the
     assumptions of Theorem \ref{thm:aim1}, and it follows from this
     classification that the linear system $|L-3p|$ will contain the
     fibre of the ruling  through $p$ precisely with multiplicity two
     as a fixed component. In particular, the map $\beta$ will
     automatically be generically finite.
   \end{remark}

   Our method is based on the observation that, when $(S,L)$ is
   triple-point defective, then at a general point 
   $p\in S$ there exists a non-reduced scheme $Z_p$ supported at the point,
   such that 
   \begin{displaymath}
     h^1\big({S},\kj_{Z_p}(L)\big)\not=0.
   \end{displaymath}
   By Serre's construction, this yields the existence of a rank $2$ bundle
   $\ke_p$ with first Chern class $L-K$, with a global section whose zero-locus 
   is a subscheme of length at most $4$, supported at $p$.
   Moreover the assumption $(L-K)^2>16$ implies that $\ke_p$ is
   \emph{Bogomolov unstable}, thus it has a destabilizing divisor $A$.
   By exploiting the properties of $A$ and $B=L-K-A$, we obtain the result.
   \smallskip
   
   In a sort of sense, one of the main points missing for the proof of Segre's conjecture
   is a natural geometric construction for the non--reduced divisor which must be
   part of any defective linear system. 

   For double-point defective surfaces,
   the non--reduced component comes from contact loci of hyperplanes (see \cite{ChC02}).

   In our setting, the non--reduced component is essentially given by the 
   effective divisor $B$ above, which comes from a destabilizing divisor
   of the rank $2$ bundle.

  The result, applied to the blowing up of $\pp^2$,
   leads to the following partial proof of Segre's conjecture on defective
   linear systems in the plane.
   
   \begin{corollary}
     Fix multiplicities  $m_1\leq m_2\leq \dots\leq m_n$.
     Let $H$ denote the class of a line in $\pp^2$ and assume that, for
     $p_1,\dots,p_n$ general in $\pp^2$, the linear
     system $M=rH-m_1p_1-\dots-m_np_n$ is defective, i.e.
     \begin{displaymath}
       \dim|M|> \max\left\{-1,\binom{r+2}2-\sum_{i=1}^n\binom{m_i+1}2\right\}.
     \end{displaymath} 
     Let $\pi:S\longrightarrow\pp^2$ be the blowing up of $\pp^2$ at the
     points $p_2,\dots,p_n$ and set 
     $L:= r\pi^*H-m_2E_2-\dots-m_nE_n$, where $E_i=\pi^*(p_i)$ is the
     i-th exceptional divisor. Assume that $L$ is very ample on $S$, 
     of the expected dimension  $\binom{r+2}2-\sum_{i=2}^n\binom{m_i+1}2$, and 
     that also $L-K$ is very ample on $S$, with $(L-K)^2>16$.
     Assume, finally, $m_1\leq 3$. 

     Then $m_1=3$ and the general element of $M$
     is non-reduced. Moreover $L$ embeds $S$ as a ruled surface.
   \end{corollary}
   \begin{proof} 
     Just apply the Main Theorem \ref{thm:aim1} to the pair $(S,L)$.
   \end{proof}

   The reader can easily check that the previous result is exactly the translation
   of Segre's and Harbourne--Hirschowitz's conjectures on defective linear 
   systems in the plane, for the case of a \emph{minimally} defective
   system with lower multiplicity $3$. The $(-1)$--curve predicted by
   Harbourne--Hirschowitz conjecture, in this situation, is just the
   pull-back of a line of the ruling.

   Although the conditions ``$L$ and $L-K$ very ample'' is not
   mild, we believe that the previous result could strengthen our believe
   in the general conjecture. Combining results in \cite{Xu95} 
   and \cite{Laz97} Corollary.\ 2.6 one can give numerical conditions on
   $r$ and the $m_i$ such that $L$ respectively $L-K$ are very ample. 

   \smallskip
   The paper is organized as follows.

   The case where $\beta$ is not generically finite 
   is pointed out in Theorem \ref{thm:notfinite} in Section
   \ref{sec:triplecomponents}. 
   In Section \ref{sec:equimultiple} we reformulate the problem as an
   $h^1$-vanishing problem.  The Sections 
   \ref{sec:construction} to \ref{sec:regular}  are
   devoted to the proof of the main result: in Section \ref{sec:construction} we
   use Serre's construction and Bogomolov instability in order to show
   that triple-point defectiveness leads to the existence of very
   special divisors $A$ and $B$ on our surface; in Section \ref{sec:zero} we
   show that $|B|$ has no fixed component; in Section \ref{sec:generalcase} we then
   list properties of $B$ and we use these in Section \ref{sec:regular} to classify
   the regular triple-point defective surfaces.

   \section{Triple components}\label{sec:triplecomponents}

   In this section, we consider what happens when, in Diagram \eqref{VVV},
    the general fiber of $\beta$ is positive-dimensional,
   in other words, when  the general member of $V_3$ contains a triple
   component through $p$. 

   This case has been investigated (and essentially solved) in \cite{Cas22}, 
   and then re\-phrased in modern language in  \cite{FrI01} and \cite{BoC05}. 
   
   Although not strictly necessary for the sequel, as our arguments 
   do not make any use of the generic finiteness of $\beta$, 
   (and so we will not assume this), for the sake of completeness
   we recall in this section some example and 
   the classification of pairs $(S,L)$
   which are triple-point defective, and such that a general curve $L_p\in|L-3p|$
   has a triple component through $p$.
   \smallskip

   The family $\kl_3$ of pairs $(L,p)\in |L|\times S$ where $L\in |L-3p|$
   has dimension bounded below by $\dim|L|-4$, and in Remark \ref{L3}
   it has been pointed out that $(S,L)$ is triple-point defective exactly when $\alpha$ is
   dominant and the bound is not attained.
   
   Notice however that $\dim|L|-4$ is {\it not} necessarily a bound for the dimension
   of the subvariety $V_3\subset |L|$, the image of $\kl_3$ under $\beta$.
   The following example (exploited in \cite{LaM02}) shows that one may have $\dim(V_3)<\dim|L|-4$ 
   even when $(S,L)$ is \emph{not} triple-point defective.
  
   \begin{example}[(see \cite{LaM02})]
      Let $S$ be the blowing up
     of $\pp^2$ at $8$ general points $q_1,\dots,q_8$ and $L$ corresponds to the system
     of curves of degree nine in $\pp^2$, with a triple point at each $q_i$.

     $\dim|L|=6$, but for $p\in S$ general, the unique divisor in
     $|L-3p|$ coincides with the cubic plane curve through $q_1,\dots,q_8,p$,
     counted three times. As there exists only a (non-linear) $1$-dimensional family
     of such divisors in $|L|$, then $\dim(V_3)=1<\dim|L|-4$. On the other hand,
     these divisors have a triple component, so that the general fibre of
     $\beta$ has dimension $1$, hence $\dim(\kl_3)=2=\dim|L|-4$.
   \end{example}
   
   The classification of triple-point defective pairs $(S,L)$ for which the map $\beta$ is not generically
   finite is the following.

   \begin{theorem} \label{thm:notfinite}
     Suppose that $(S,L)$ is triple-point defective. Then for $p\in S$ general, 
    the general member of $|L-3p|$
    contains a triple  component through $p$ if and only if
     $S$ lies in a threefold $W$ which is a scroll in planes and moreover $W$ is
     developable, i.e. the tangent space to $W$ is constant along the planes.
   \end{theorem}
   \begin{proof} (HINT) 
     First, since we assume that $S$ is triple-point defective and
     embedded in $\pp^r$ via $L$, then the hyperplanes
     $\pi$ that meet $S$ in a divisor $H=S\cap \pi$ with a triple point at a general
     $p\in S$, intersect in a $\pp^4$. Thus we may project down $S$ to $\pp^5$ and work
     with the corresponding surface.

     In this setting, through a general $p\in S$ one has only one hyperplane $\pi$ with
     a triple contact, and $\pi$ has a triple contact with $S$ along the fibre $C$ of
     $\beta$. Thus $V_3$ is a curve.

     If $H', H''$ are two consecutive infinitesimally near points to $H$ on $V_3$, then
     $C$ also belongs to $H\cap H'\cap H''$. Thus $C$ is a plane curve and $S$ is fibred
     by a $1$-dimensional family of plane curves. This determines the threefold scroll $W$.

     The tangent line to $V_3$ determines in $(\pp^5)^*$ a pencil of hyperplanes which are
     tangent to $S$ at any point of $C$, since this is the infinitesimal deformation
     of a family of hyperplanes with a triple contact along any point of $C$. Thus
     there is a $\pp^4=H_C$ which is tangent to $S$ along $C$.

     Assume that $C$ is not a line. Then $C$ spans a $\pp^2=\pi_C$ fibre of $W$, moreover the
     tangent space to $W$ at a general point of $C$ is spanned by $\pi_C$ and $T_{S,P}$, hence
     it is constantly equal to $H_C$. Since $C$ spans $\pi_C$, then it turns out that
     the tangent space to $W$ is constant at any point of $\pi_C$, i.e. $W$ is developable.

     When $C$ is a line, then arguing as above one finds that all the  tangent planes
     to $S$ along $C$ belong to the same $\pp^3$. This is enough to conclude
     that $S$ sits in some developable $3$-dimensional scroll.

     Conversely, if $S$ is contained in the developable scroll $W$, then at a general point
     $p$, with local coordinates $x,y$, the tangent space $t$ to $W$ at $p$ contains the
     derivatives $p, p_x, p_y, p_{xx}, p_{xy}$ (here $x$ is the direction of the tangent
     line to $C$). Thus the $\pp^4$ spanned by $t,p_{yy}$ intersects $S$ in a triple curve
     along $C$.
   \end{proof}

   \section{The Equimultiplicity Ideal}\label{sec:equimultiple}

   If $L_p$ is a curve in $|L-3p|$ we denote by $f_p\in\C\{x_p,y_p\}$ an
   equation of $L_p$ in local coordinates $x_p$ and $y_p$ at $p$.
   If $\mult_p(L_p)=3$, the ideal sheaf $\kj_{Z_p}$ whose stalk 
   at $p$ is the equimultiplicity   ideal
   \begin{displaymath}
     \kj_{Z_p,p}=\left\langle\frac{\partial
         f_p}{\partial x_p},\frac{\partial f_p}{\partial
         y_p}\right\rangle + \langle x_p,y_p\rangle^3
   \end{displaymath}
   of $f_p$ defines a zero-dimensional scheme
   $Z_p=Z_p(L_p)$ concentrated at $p$, and the tangent space
   $T_{(L_p,p)}({\kl_3})$ of $\kl_3$ at $(L_p,p)$ satisfies (see
   \cite{Mar06} Example 10)
   \begin{displaymath}
     T_{(L_p,p)}({\kl_3})\cong \big(H^0\big(S,\kj_{Z_p}(L_p)\big)/H^0(S,\ko_S)\big)\oplus\kk,
   \end{displaymath}
   where $\kk$ is zero unless $L_p$ is unitangential at $p$, in which
   case $\kk$ is a one-dimensional  vector space.

   In particular, $\kl_3$ is smooth at $(L_p,p)$ of the expected
   dimension (see \cite{Mar06} Proposition 11)
   \begin{displaymath}
     \expdim(\kl_3)=\dim|L|-4
   \end{displaymath}
   as soon as
   \begin{displaymath}
     h^1\big({S},\kj_{Z_p}(L)\big)=0.
   \end{displaymath}
   We thus have the following proposition.

   \begin{proposition}
     Let $S$ be regular and suppose that $\alpha$ is surjective,
     then $(S,L)$ is not triple-point defective if 
     \begin{displaymath}
       h^1\big({S},\kj_{Z_p}(L)\big)=0
     \end{displaymath}
     for general $p\in S$ and $L_p\in|L|$ with $\mult_p(L_p)=3$.

     Moreover, if $L$ is non-special the above $h^1$-vanishing is also
     necessary for the non-triple-point-defectiveness of $(S,L)$.
   \end{proposition}

   \section{The Basic Construction}\label{sec:construction}

   \medskip
   \begin{center}
     \framebox[12cm]{
       \begin{minipage}{11.3cm}
         \medskip
         \emph{From now on we assume that for $p\in S$ general
           $\exists\;L_p\in|L|$ s.t.
           $$h^1\big({S},\kj_{Z_p}(L)\big)\not=0.$$}
         \vspace*{-2ex}
       \end{minipage}
       }
   \end{center}
   \medskip

   Then by Serre's construction for a subscheme $Z'_p\subseteq Z_p$ with
   ideal sheaf $\kj_p=\kj_{Z_p'}$ of minimal length such that
   $h^1\big(S,\kj_p(L)\big)\not=0$ there is a rank two bundle $\ke_p$
   on $S$ and a section $s\in H^0(S,\ke_p)$ whose $0$-locus is $Z'_p$,
   giving the exact sequence
   \begin{equation}\label{eq:vectorbundle}
     0\rightarrow
     \ko_S\rightarrow\ke_p\rightarrow\kj_p(L-K)\rightarrow 0.
   \end{equation}
   The Chern classes of $\ke_p$ are
   \begin{displaymath}
     c_1(\ke_p)=L-K\;\;\;\mbox{ and }\;\;\;
     c_2(\ke_p)=\length(Z'_p).
   \end{displaymath}
   Moreover, $Z'_p$ is automatically a complete
   intersection.

   We would now like to understand what $\kj_p$ is depending on
   $\jet_3(f_p)$, which in suitable local coordinates will be one
   of those in Table \eqref{eq:3jets}.
   For this we first of all note that the very ample
   divisor $L$ separates all subschemes of $Z_p$ of length
   at most two. Thus $Z'_p$ has length at least $3$, and due to Lemma
   \ref{lem:ideals} below we are in one of the following situations:

   \begin{equation}\label{eq:3jets}
     \begin{array}{|c|c|c|c|c|}
       \hline
       \jet_3(f_p)&\kj_{Z_p,p}&\length(Z_p)&\kj_p&c_2(\ke_p)\\
       \hline\hline
       x_p^3-y_p^3 & \langle x_p^2,y_p^2 \rangle & 4 & \langle x_p^2,y_p^2 \rangle & 4\\\hline
       x_p^2y_p  &\langle x_p^2, x_py_p,y_p^3\rangle & 4 &\langle x_p,y_p^3 \rangle & 3\\\hline
       x_p^3 & \langle x_p^2,x_py_p^2,y_p^3\rangle & 5&\langle x_p^2,y_p^2 \rangle & 4\\\hline
       x_p^3 & \langle x_p^2,x_py_p^2,y_p^3\rangle & 5&\langle x_p,y_p^3 \rangle & 3\\\hline
     \end{array}
   \end{equation}

   \bigskip

   \begin{lemma}\label{lem:ideals}
     If $f\in R=\C\{x,y\}$ with $\jet_3(f)\in\{x^3-y^3,x^2y,x^3\}$, and if $I=\langle
     g,h\rangle\lhd R$ such that
     $\dim_\C(R/I)\geq 3$ and $\big\langle \frac{\partial f}{\partial
       x},\frac{\partial f}{\partial y}\big\rangle+\langle
     x,y\rangle^3\subseteq I$, then we may assume that we are in
     one of the following cases:
     \begin{enumerate}
     \item $I=\langle x^2,y^2\rangle$ and
       $\jet_3(f)\in\{x^3-y^3,x^3\}$, or
     \item $I=\langle x,y^3\rangle$ and $\jet_3(f)\in\{x^2y,x^3\}$.
     \end{enumerate}
   \end{lemma}
   \begin{proof}
     If $>$ is any \emph{local degree} ordering on $R$, then the
     Hilbert-Samuel functions of $R/I$ and of $R/L_>(I)$ coincide,
     where $L_>(I)$ denotes the leading ideal of $I$ (see
     e.g. \cite{GrP02} Proposition 5.5.7). In particular,
     $\dim_\C(R/I)=\dim_\C(R/L_>(I))$ and thus
     \begin{displaymath}
       L_>(I)\in\big\{\langle x^2,xy^2,y^3\rangle,\langle x^2,xy,y^2\rangle,\langle
       x^2,xy,y^3\rangle,\langle x^2,y^2\rangle,\langle
       x,y^3\rangle\},
     \end{displaymath}
     since $\langle x^2,xy^2,y^3\rangle\subset I$.

     Taking $>$, for a moment, to be the local
     degree ordering on $R$ with $y>x$ we deduce at once that $I$ does
     not contain any power series with a linear term in $y$. For the
     remaining part of the proof $>$ will be the local degree ordering
     on $R$ with $x>y$.

     \underline{1st Case:} $L_>(I)=\langle x^2,xy^2,y^3\rangle$ or 
    $L_>(I)=\langle x^2,xy,y^2\rangle$.
       Thus the graph of the slope $H^0_{R/I}$ of the Hilbert-Samuel  function of $R/I$
       would be as shown in Figure \ref{fig:fp-histogram}, which
       contradicts the fact that $I$ is a complete intersection due to
       \cite{Iar77} Theorem 4.3.
     \begin{figure}[h]
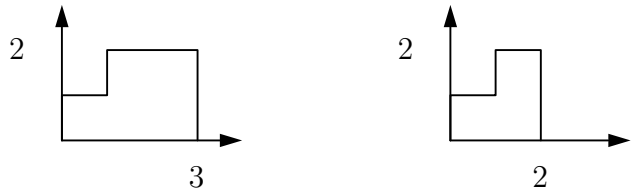

       \smallskip
       \begin{center}
         \begin{texdraw}
           \drawdim cm
           \setunitscale 0.6
           \arrowheadtype t:F
           \arrowheadsize l:0.5 w:0.3
           \move (0 0)
           \rlvec (0 1) \rlvec (1 0)
           \rlvec (0 1) \rlvec (2 0)
           \rlvec (0 -2)
           \move (0 0) \avec (0 3)
           \move (0 0) \avec (4 0)
           \move (0 0)
           \rmove (3 -0.6) \textref h:C v:T \htext{$3$}
           \move (-0.8 2)  \textref h:R v:C \htext{$2$}
         \end{texdraw}
         \hspace*{2cm}
         \begin{texdraw}
           \drawdim cm
           \setunitscale 0.6
           \arrowheadtype t:F
           \arrowheadsize l:0.5 w:0.3
           \move (0 0)
           \rlvec (0 1) \rlvec (1 0)
           \rlvec (0 1) \rlvec (1 0)
           \rlvec (0 -2)
           \move (0 0) \avec (0 3)
           \move (0 0) \avec (4 0)
           \move (0 0)
           \rmove (2 -0.6) \textref h:C v:T \htext{$2$}
           \move (-0.8 2)  \textref h:R v:C \htext{$2$}
         \end{texdraw}
         \medskip
         \caption{The graphs of $H^0_{R/\langle x^2,xy^2,y^3\rangle}$
           respectively of $H^0_{R/\langle x^2,xy,y^2\rangle}$.}
         \label{fig:fp-histogram}
       \end{center}
     \end{figure}

     \underline{2nd Case:} $L_>(I)=\langle x^2,xy,y^3\rangle$.
     Then we may assume
     \begin{displaymath}
       g=x^2+\alpha\cdot y^2+h.o.t.\;\;\;\mbox{ and }\;\;\; h=xy+\beta\cdot y^2+h.o.t..
     \end{displaymath}
     Since $x^2\in I$ there are power series $a,b\in R$ such that
     \begin{displaymath}
       x^2=a\cdot g+b\cdot h.
     \end{displaymath}
     Thus the leading monomial of $a$ is one, $a$ is a unit and
     $g\in\langle x^2,h\rangle$. We may therefore assume that
     $g=x^2$. Moreover, since the intersection multiplicity of $g$ and
     $h$ is $\dim_\C(R/I)=4$, $g$ and $h$ cannot have a common tangent
     line in the origin, i.\ e.\ $\beta\not=0$. Thus, since $g=x^2$,
     we may assume that $h=xy+y^2\cdot u$ with $u=\beta+h.o.t$ a unit.

     In new coordinates $\widetilde{x}=x\cdot\sqrt{u}$ and
     $\widetilde{y}=y\cdot\frac{1}{\sqrt{u}}$ we have
     \begin{displaymath}
       I=\langle \widetilde{x}^2,\widetilde{x}\widetilde{y}+\widetilde{y}^2\rangle.
     \end{displaymath}
     Note that by the coordinate change $\jet_3(f)$ only changes by a constant,
     that $\frac{\partial f}{\partial \widetilde{x}},\frac{\partial
       f}{\partial \widetilde{y}}\in I$ and that $\langle
     \widetilde{x},\widetilde{y}\rangle^3\subset I$, but
     $\widetilde{x}\widetilde{y},\widetilde{y}^2\not\in I$. Thus
     $\jet_3(f)=x^3$.

     Setting now $\bar{x}=\widetilde{x}$ and
     $\bar{y}=\widetilde{x}+2\widetilde{y}$, then
     $\bar{y}^2=\widetilde{x}^2+4\cdot
     (\widetilde{x}\widetilde{y}+\widetilde{y}^2)\in I$ and thus,
     considering colengths,
     \begin{displaymath}
       I=\langle \bar{x}^2,\bar{y}^2\rangle.
     \end{displaymath}
     Moreover, the $3$-jet of $f$ does not change with respect
     to the new coordinates, so that we may assume we worked with
     these from the beginning.

     \underline{3rd Case:} $L_>(I)=\langle x^2,y^2\rangle$.
     Then we may assume
     \begin{displaymath}
       g=x^2+\alpha\cdot xy+h.o.t.\;\;\;\mbox{ and }\;\;\; h=y^2+h.o.t.
     \end{displaymath}
     As in the second case we deduce that w.l.o.g.\ $g=x^2$ and
     thus $h=y^2\cdot u$, where $u$ is a unit. But then $I=\langle
     x^2,y^2\rangle$.

     \underline{4th Case:} $L_>(I)=\langle x,y^3\rangle$.
     Then we may assume
     \begin{displaymath}
       g=x+h.o.t.\;\;\;\mbox{ and }\;\;\; h=y^3+h.o.t.
     \end{displaymath}
     since there is no power series in $I$ involving a linear term in
     $y$. In new coordinates $\widetilde{x}=g$ and
     $\widetilde{y}=y$ we have
     \begin{displaymath}
       I=\big\langle \widetilde{x},\widetilde{h}\big\rangle,
     \end{displaymath}
     and we may assume that $\widetilde{h}=\widetilde{y}^3\cdot u$, where $u$ is a
     unit only depending on $\widetilde{y}$. Hence, $I=\langle
     \widetilde{x},\widetilde{y}^3\rangle$. Moreover, the $3$-jet
     of $f$ does not change with respect
     to the new coordinates, so that we may assume we worked with
     these from the beginning.
   \end{proof}

   \medskip
   \begin{center}
     \framebox[11cm]{
       \begin{minipage}{10cm}
         \medskip
         From now on we assume that $(L-K)^2>16$.
         \medskip
       \end{minipage}
       }
   \end{center}
   \bigskip

   Thus
   \begin{displaymath}
     c_1(\ke_p)^2-4\cdot c_2(\ke_p)>0,
   \end{displaymath}
   and hence $\ke_p$ is Bogomolov unstable. The Bogomolov instability
   implies the existence of a unique divisor $A_p$ which destabilizes
   $\ke_p$. (See e.\ g.\ \cite{Fri98} Section 9, Corollary 2.) In
   other words, setting $B_p=L-K-A_p$, i.\ e.\
   \begin{equation}
     \label{eq:AB:0}
     A_p+B_p=L-K,
   \end{equation}
   there is an immersion
   \begin{equation}
     0\rightarrow\ko_S(A_p)\rightarrow\ke_p
   \end{equation}
   where $(A_p-B_p)^2\geq c_1(\ke_p)^2-4\cdot c_2(\ke_p)>0$ and 
    $(A_p-B_p).H>0$ for every ample $H$. From
   this we deduce the following properties:
   \begin{enumerate}
   \item $\ke_p(-A_p)$ has a section that vanishes along a
     subscheme $\widetilde{Z}_p$ of codimension $2$, and
     \begin{equation}
       \label{eq:AB:2}
       A_p.B_p\leq\length(Z'_p).
     \end{equation}
    The previous immersion gives rise to a short exact sequence:
  \begin{equation}\label{eq:AB:1}
     0\rightarrow\ko_S(A_p)\rightarrow\ke_p\rightarrow\kj_{\widetilde{Z}_p}(B_p)\rightarrow 0.
   \end{equation}
   \item The divisor $B_p$ is effective and we may assume that
     $Z'_p\subset B_p$.
   \item $A_p-B_p$ is big, and hence
     $\dim\big(|k\cdot(A_p-B_p)|\big)=const\cdot k^2+o(k)>0$
     for $k>>0$. In particular
     \begin{equation}
       \label{eq:AB:3}
       (A_p-B_p).M>0
     \end{equation}
     if $M$ is big and nef or if $M$ is an irreducible curve with
     $M^2\geq 0$ or if $M$ is effective without fixed part.
   \item $A_p$ is big.
   \end{enumerate}
   \begin{proof}
     \begin{enumerate}
     \item Sequence \eqref{eq:AB:1} is a consequence of Serre's construction\tom{ 
    (see \cite{Fri98} loc. cit.)}.
    The first assertion now follows from Sequence
       \eqref{eq:AB:1}, and Equation \eqref{eq:AB:2} is a
       consequence of
       \begin{displaymath}
         (A_p-B_p)^2\geq c_1(\ke_p)^2-c_2(\ke_p)=(A_p+B_p)^2-4\cdot\length(Z'_p).
       \end{displaymath}
     \item Observe that $\big(2A_p-(L-K)\big).H>0$ for any ample line
       bundle $H$, and thus
       \begin{displaymath}
         -A_p.H<-\frac{(L-K_{S}).H}{2}< 0.
       \end{displaymath}
       Thus $H^0\big(\ko_S(-A_p)\big)=0$ and twisting the sequence
       \eqref{eq:vectorbundle} with $-A_p$ we are done.
     \item Since $(A_p-B_p)^2>0$ and $(A_p-B_p).H>0$ for some ample
       $H$ Riemann-Roch shows that $A_p-B_p$ is big, i.\ e.\
       $\dim\big(|k\cdot(A_p-B_p)|\big)$ grows with $k^2$. The
       remaining part follows from Lemma \ref{lem:big}.
     \item This follows since $A_p-B_p$ is big and $B_p$ is effective.
     \end{enumerate}
   \end{proof}

   \begin{lemma}\label{lem:big}
     Let $R$ be a big divisor. If $M$ is big and nef or if $M$ is an irreducible curve with
     $M^2\geq 0$ or if $M$ is an effective divisor without fixed component, then $R.M>0$.
   \end{lemma}
   \begin{proof}
     If $R$ is big, then $\dim|k\cdot R|$ grows with $k^2$.
     Thus for
     $k>>0$ we can write $k\cdot R=N'+N''$ where $N'$ is ample
     and $N''$ effective (possibly zero). To see this, note that
     for $k>>0$ we can write $|k\cdot R|=|N'|+N''$, where $N''$ is the
     fixed part of $|k\cdot R|$ and $N'\cap C\not=\emptyset$ for every
     irreducible curve $C$. Then apply the Nakai-Moishezon Criterion
     to $N'$ (see also \cite{Tan04}). 

     Analogously, if $M$ is big and nef, for
     $j>>0$ we can write $j\cdot M=M'+M''$ where $M'$ is ample and
     $M''$ is effective. Therefore,
     \begin{displaymath}
       R.M=\frac{1}{kj}\cdot \big(N'.M'+N'.M''+N''.M)>0,
     \end{displaymath}
     since $N'.M'>0$, $N'.M''\geq 0$ and $N''.M\geq 0$.

     Similarly, if $M$ is irreducible and has non-negative
     self-inter\-section, then
     \begin{displaymath}
       R.M=\frac{1}{k}\cdot(N'.M+N''.M)>0.
     \end{displaymath}

     And if $M$ is effective without fixed component, we can apply the
     previous argument to every component of $M$.
   \end{proof}

   Now let $p$ move freely in $S$. Accordingly the scheme
   $Z'_p$ moves, hence the effective divisor $B_p$ containing $Z'_p$
   moves in an algebraic family $\kb\subseteq |B|_a$ which is the
   closure of $\{B_p\;|\;p\in S, L_p\in|L-3p|, \mbox{ both general}\}$
   and which covers $S$. 
   A priori this family $\kb$ might
   have a \emph{fixed part} $C$, so that for general $p\in S$ there is an
   effective divisor $D_p$ moving in a fixed-part free algebraic
   family $\kd\subseteq |D|_a$  such that
   \begin{displaymath}
     B_p=C+D_p.
   \end{displaymath}
   Whenever we only refer to the algebraic class of $A_p$ respectively
   $B_p$ respectively $D_p$ we will write $A$ respectively $B$
   respectively $D$ for short.

   \medskip
   \begin{center}
     \framebox[11cm]{
       \begin{minipage}{10cm}
         \medskip
         For these considerations we assume, of course, that $\length(Z'_p)$ is constant for
         $p\in S$ general, so either $\length(Z'_p)=3$ or $\length(Z'_p)=4$.
         \medskip
       \end{minipage}
       }
   \end{center}
   \bigskip

   \section{$C=0$.}\label{sec:zero}

   Our first aim is to show that actually $C=0$ (see Lemma
   \ref{lem:C}). But in order to do so
   we first have to consider the
   boundary case that $A_p.B_p=\length(Z_p')$.

   \begin{proposition}\label{prop:splitting}
     If $A_p.B_p=\length(Z'_p)$, then
     there exists a non-trivial global section $0\not=s\in
     H^0\big(B_p,\kj_{Z_p'/B_p}(A_p)\big)$ whose zero-locus is
     $Z_p'$.

     In particular, $A_p.D_p=A_p.B_p=\length(Z_p')$ and $A_p.C=0$.
   \end{proposition}
   \begin{proof}
     By Equation \eqref{eq:AB:1} we have
     \begin{displaymath}
       A_p.B_p=\length(Z'_p)=c_2(\ke_p)=A_p.B_p+\length\big(\widetilde{Z}_p\big).
     \end{displaymath}
     Thus $\widetilde{Z}_p=\emptyset$.

     If we  merge the sequences \eqref{eq:vectorbundle},
     \eqref{eq:AB:1}, and the structure sequence of $B$ twisted by $B$
     we obtain the following exact commutative diagram in Figure~\ref{fig:AB},
     \begin{figure}[h]
       \begin{displaymath}
         \xymatrix{
           && 0 \ar[d] & 0 \ar[d] &\\
           &0\ar[r]\ar[d]& \ko_S \ar[r]\ar[d] & \ko_S \ar[r]\ar[d] &0 \\
           0\ar[r]& \ko_S(A_p) \ar[r]\ar[d] & \ke_p \ar[r]\ar[d]&\ko_S(B_p)\ar[r]\ar[d] &0 \\
           0\ar[r]& \ko_S(A_p) \ar[r]\ar[d] & \kj_{Z_p'/S}(A_p+B_p) \ar[r]\ar[d]&\ko_{B_p}(B_p)\ar[r]\ar[d] &0 \\
           &0&0&0
         }
       \end{displaymath}\centering

       \caption{The diagram showing $\ko_{B_p}=\kj_{Z_p'/B_p}(A_p)$.}
       \label{fig:AB}
     \end{figure}
     where $\ko_{B_p}(B_p)=\kj_{Z_p'/B_p}(A_p+B_p)$, or equivalently
     $\ko_{B_p}=\kj_{Z_p'/B_p}(A_p)$.
     Thus from the rightmost column we get a non-trivial global
     section, say $s$, of this bundle 
     which vanishes precisely at $Z_p'$,
     since $Z_p'$ is the zero-locus of the monomorphism of vector
     bundles $\ko_S\hookrightarrow \ke_p$. However, since $p$ is
     general we have that $p\not\in C$ and thus the restriction
     $0\not=s_{|D_p}\in H^0\big(D_p,\kj_{Z_p'/D_p}(A_p)\big)$ and it still vanishes
     precisely at $Z_p'$. Thus $A_p.D_p=\length(Z_p')=A_p.B_p$, and $A_p.C=A_p.B_p-A_p.D_p=0$.
   \end{proof}

       




   \begin{lemma}\label{lem:AB}
     $A_p.B_p\geq 1+B_p^2$.
   \end{lemma}
   \begin{proof}
     Let $B=P+N$ be a Zariski decomposition of $B$, i.\ e.\ $P$ and
     $N$ are effective $\Q$-divisors such that in particular $P$ is nef, $P.N=0$
     and $N^2<0$ unless $N=0$.

     If $N\not=0$, then
     \begin{displaymath}
       0<(A+B).N=A.N+N^2,
     \end{displaymath}
     since $A+B$ is very ample and $N$ is effective. Moreover, since $P$
     is nef and $A-B$ big we have $(A-B).P\geq 0$ and hence
     \begin{displaymath}
       A.P\geq B.P=P^2.
     \end{displaymath}
     Combining these two inequalities we get
     \begin{displaymath}
       A.B=A.P+A.N>P^2-N^2>P^2+N^2=B^2.
     \end{displaymath}

     If $N=0$, then $B$ is nef and, therefore, $B^2\geq 0$. If,
     actually $B^2>0$, then $B$ is big and nef, so that by
     \eqref{eq:AB:3} $(A-B).B>0$. While if $B^2=0$ then
     \begin{displaymath}
       B^2=0<B.(A+B)=A.B
     \end{displaymath}
     since $A+B$ is very ample and $B$ is effective.
   \end{proof}

   \begin{lemma}\label{lem:DD}
     Let $p\in S$ be general and suppose $\length(Z'_p)=4$.
     \begin{enumerate}
     \item If $D_p$ is irreducible, then $\dim(\kd)\geq 2$ and
       $D_p^2\geq 3$.
     \item If $D_p$ is reducible but the part containing $p$ is reduced,
       then
       either $D_p$ has a component singular in $p$ and $D_p^2\geq 5$
       or at least two components of $D_p$ pass through $p$ and
       $D_p^2\geq 2$.
     \item If $D_p^2\leq 1$, then $D_p=k\cdot E_p$ where $k\geq 2$,
       $E_p$ is irreducible and $E_p^2=0$. In particular, $D_p^2=0$.
     \end{enumerate}
   \end{lemma}
   \begin{proof}
     \begin{enumerate}
     \item
       If $D_p$ is irreducible, then $\dim(\kd)\geq 2$, since $D_p$,
       containing $Z'_p$, is
       singular in $p$ by Table \eqref{eq:3jets} and since $p\in S$ is
       general. If through $p\in S$
       general and a general $q\in D_p$ there is another $D'\in\kd$,
       then due to the irreducibility of $D_p$
       \begin{displaymath}
         D_p^2=D_p.D'\geq \mult_p(D_p)+\mult_q(D_p)\geq 3.
       \end{displaymath}
       Otherwise, $\kd$ is a two-dimensional involution whose
       general element is irreducible, so that by \cite{ChC02} Theorem
       5.10 $\kd$ must be a linear system. This, however,
       contradicts the Theorem of Bertini, since the general element
       of $\kd$ would be singular.
     \item
       Suppose $D_p=\sum_{i=1}^k E_{i,p}$ is reducible but the part
       containing $p$ is reduced. Since
       $D_p$ has no fixed component and $p$ is general, each $E_{i,p}$
       moves in an at least one-dimensional family. In particular
       $E_{i,p}^2\geq 0$.

       If some $E_{i,p}$, say $i=1$, would be singular in $p$ for
       $p\in S$ general we could argue as above that $E_{1,p}^2\geq
       3$. Moreover, either $E_{2,p}$ is algebraically equivalent to
       $E_{1,p}$ and $E_{2,p}^2\geq3$, or $E_{1,p}$ and $E_{2,p}$
       intersect properly, since both vary in different, at least
       one-dimensional families. In any case we have
       \begin{displaymath}
         D_p^2\geq (E_{1,p}+E_{2,p})^2\geq 5.
       \end{displaymath}
       Otherwise, at least two components, say
       $E_{1,p}$ and $E_{2,p}$ pass through $p$, since $D_p$ is singular
       in $p$ and no component passes through $p$ with higher
       multiplicity. Hence, $E_{1,p}.E_{2,p}\geq 1$ and therefore
       \begin{displaymath}
         D_p^2\geq 2\cdot E_{1,p}.E_{2,p}\geq 2.
       \end{displaymath}
     \item
       From the above we see that $D_p$ is not reduced in $p$. Let therefore
       $D_p\equiv_a kE_p+E'$
       where $k\geq 2$, $E_p$ passes through $p$ and $E'$ does not contain any
       component algebraically
       equivalent to $E_p$\tom{ (note, if $E'$ contains a component
         algebraically equivalent to $E_p$ we may replace it by $E_p$
         without changing anything)}.

       Suppose $E'\not=0$.\label{eq:DD:0} Since $D_p$ has
       no fixed component both, $E_p$ and $E'$ vary in an at least one
       dimensional family covering $S$ and must therefore intersect
       properly. In particular, $E_p.E'\geq 1$ and $1\geq D_p^2\geq 2k\cdot
       E_p.E'\geq 4$. Thus, $E'=0$.

       We therefore may assume that
       $D_p=kE_p$ with $k\geq 2$. Then $0\leq E_p^2=\frac{1}{k^2}\cdot
       D_p^2\leq \frac{1}{4}$, which leaves only the possibility $E_p^2=0$,
       implying also $D_p^2=0$.
     \end{enumerate}
   \end{proof}

   \begin{lemma}\label{lem:curves}
     Suppose that $R\subset S$ is an irreducible curve.
     \begin{enumerate}
     \item If $(L-K).R\in\{1,2\}$, then $R$ is smooth, rational and
       $R^2\leq (L-K).R-3\leq -1$.
     \item If $(L-K).R=3$, then $R^2\leq 0$, and either $R$ is a plane cubic or it is a
       smooth rational space  curve.
     \end{enumerate}
   \end{lemma}
   \begin{proof}
     Note that $S$ is embedded in some $\pp^n$ via $|L-K|$ and that
     $\deg(R)=(L-K).R$ is just the degree of $R$ as a curve in
     $\pp^n$. Moreover, by the adjunction formula we know that
     \begin{displaymath}
       p_a(R)=\frac{R^2+R.K}{2}+1,
     \end{displaymath}
     and since $L$ is very ample we thus get
     \begin{equation}
       \label{eq:curves:1}
       1\leq L.R =(L-K).R+R.K=(L-K).R+2\cdot\big(p_a(R)-1\big)-R^2.
     \end{equation}
     \begin{enumerate}
     \item If $\deg(R)\in\{1,2\}$, then $R$ must be a smooth,
       rational curve. Thus we deduce from \eqref{eq:curves:1}
       \begin{displaymath}
         R^2\leq (L-K).R-3.
       \end{displaymath}
     \item If $\deg(R)=3$, then $R$ is either a plane cubic or a
       smooth space curve of genus $0$\tom{ (see e.\ g.\ \cite{Har77}
         Ex.\ IV.3.4; if $H\subseteq \pp^n$ is a linear subspace of
         minimal dimension $m>2$ containing $R$ then through $m$
         arbitrary points $p_1,\ldots,p_m$ on $R$ there is a linear subspace in $H$ and
         $H.R\geq \sum_{i=1}^m \mult_{p_i}(R)$; thus $m=3$ and $R$ is
         smooth)}.
       If $p_a(R)=1$ then actually $L.R\geq 3$ since otherwise $|L|$
       would embed $R$ as a rational curve of degree $1$ resp.\ $2$ in
       some projective space. In any case we are therefore done with
       \eqref{eq:curves:1}.
     \end{enumerate}
   \end{proof}

   \begin{lemma}\label{lem:C}
     $C=0$.
   \end{lemma}
   \begin{proof}
     Suppose $C\not=0$ and $r$ is the number of irreducible components
     of $C$. Since $\kd$ has no fixed
     component we know by \eqref{eq:AB:2} that $(A-B).D>0$, so
     that
     \begin{equation}\label{eq:c:0}
       A.D\geq B.D+1=D.C+D^2+1
     \end{equation}
     or equivalently
     \begin{equation}\label{eq:c:1}
       D.C\leq A.D-D^2-1.
     \end{equation}
     Moreover, since $A+B$ is very ample we have
     $r\leq(A+B).C=A.C+D.C+C^2$ and thus
     \begin{equation}\label{eq:c:2}
       A.C+D.C=(A+B).C-C^2\geq r-C^2.
     \end{equation}

     \textbf{1st Case:} $C^2\leq 0$.
     Then \eqref{eq:c:2}
     together with \eqref{eq:c:0} gives
     \begin{equation}\label{eq:c:3}
       A.B=A.C+A.D\geq A.C+D.C+D^2+1\geq r+(-C^2)+D^2+1\geq 2,
     \end{equation}
     or the slightly stronger inequality
     \begin{equation}
       \label{eq:c:4}
       \length(Z_p')\geq A.B\geq (A+B).C +(-C^2)+D^2+1.
     \end{equation}

     \textbf{2nd Case:} $C^2>0$. Then by Lemma \ref{lem:AB} simply
     \begin{equation}\label{eq:c:5}
       \length(Z_p')\geq A.B\geq B^2+1=D^2+2\cdot C.D+C^2+1\geq 2.
     \end{equation}

     Since all the summands involved in the right hand side of \eqref{eq:c:3} and
     \eqref{eq:c:5} are non-negative and since by Lemma \ref{lem:DD}
     the case $D^2=1$ cannot occur when $\length(Z_p')=4$, we are left considering the
     cases shown in Figure~\ref{fig:c}, where for the additional information (the last
     four columns) we take
     Proposition \ref{prop:splitting} and Lemma \ref{lem:DD} into account.
     \begin{figure}[h]
       \begin{displaymath}
         \begin{array}{|c|c|c|c|c|c||c|c|c|c|}
           \hline
           & \length(Z'_p) & D^2 & C^2 & C.D & r & A.B & A.D & A.C & D
           \\\hline\hline
           1)&       4        & 0  & -2  &     & 1 &  4  &  4  &  0  & kE,k\geq 2 \\\hline
           2)&       4        & 0  & -1  &     & 2 &  4  &  4  &  0  & kE,k\geq 2 \\\hline
           3)&       4        & 0  &  0  &     & 3 &  4  &  4  &  0  & kE,k\geq 2 \\\hline
           4)&       4        & 0  & -1  &     & 1 & 3,4 &     &     & kE,k\geq 2 \\\hline
           5)&       4        & 2  &  0  &     & 1 &  4  &  4  &  0  &    \\\hline
           6)&       4        & 0  &  0  &     & 2 & 3,4 &     &     & kE,k\geq 2 \\\hline
           7)&       4        & 0  &  0  &     & 1 &2,3,4&     &     & kE,k\geq 2 \\\hline
           8)&       3        & 0  & -1  &     & 1 &  3  &  3  &  0  &    \\\hline
           9)&       3        & 0  &  0  &     & 2 &  3  &  3  &  0  &    \\\hline
           10)&       3        & 1  &  0  &     & 1 &  3  &  3  &  0  &    \\\hline
           11)&      3        & 0  &  0  &     & 1 & 2,3 &     &     &    \\\hline\hline
           12)&      4        & 0  &  1  &  1  &   &  4  &  4  &  0  & kE,k\geq 2 \\\hline
           13)&      4        & 2  &  1  &  0  &   &  4  &  4  &  0  &    \\\hline
           14)&      4        & 0  &  1  &  0  &   &2,3,4&     &     & kE,k\geq 2   \\\hline
           15)&      4        & 0  &  2  &  0  &   & 3,4 &     &     & kE,k\geq 2   \\\hline
           16)&      3        & 1  &  1  &  0  &   &  3  &  3  &  0  &    \\\hline
           17)&      3        & 0  &  1  &  0  &   & 2,3 &     &     &    \\\hline
         \end{array}
       \end{displaymath}\centering

       \caption{The cases to be considered.}
       \label{fig:c}
     \end{figure}

     Let us first and for a while consider the situation $\length(Z_p')=4$ and
     $D^2=0$, so that by Lemma \ref{lem:DD}\;\; $D=kE$ for some
     irreducible curve $E$ with $k\geq 2$ and
     $E^2=0$. Applying Lemma \ref{lem:curves} to $E$ we see that
     $(A+B).E\geq 3$, and thus
     \begin{equation}\label{eq:c:6a}
       6\leq 3k\leq (A+B).D=A.D+C.D.
     \end{equation}

     If in addition $A.D\leq 4$, then \eqref{eq:c:1} leads to
     \begin{equation}
       6\leq A.D+C.D\leq 4+C.D\leq 7,
     \end{equation}
     which is only possible for $k=2$, $C.E=1$ and
     \begin{equation}
       C.D=k\cdot C.E=2.\label{eq:c:6}
     \end{equation}
     This outrules Case 12.

     In Cases 1, 2  and 3 we have $A.D=4$, and we can apply \eqref{eq:c:6}, which
     by \eqref{eq:c:2} then gives the contradiction
     \begin{displaymath}
       2=A.C+C.D\geq r-C^2=3.
     \end{displaymath}

     If, still under the assumption $\length(Z_p')=4$ and
     $D^2=0$, we moreover assume  $2\geq C^2\geq 0$ then by Lemma \ref{lem:AB}
     \begin{displaymath}
       3\geq B^2=2\cdot C.D+C^2\geq 2\cdot C.D\geq 0,
     \end{displaymath}
     and thus $C.D\leq 1$ and $C.D+C^2\leq 2$, which due to \eqref{eq:c:6a}
     implies $A.D\geq 5$. But then by Proposition \ref{prop:splitting}
     we have $A.B\leq 3$ and hence $A.C=A.B-A.D\leq -2$, which leads
     to the contradiction
     \begin{equation}\label{eq:c:7}
       (A+B).C=A.C+D.C+C^2\leq 0,
     \end{equation}
     since $A+B$ is very ample. This outrules the Cases 6,
     7, 14 and 15.


     In Case 4 Lemma \ref{lem:curves} applied to $C$ shows
     \begin{equation}\label{eq:c:8}
       2\leq (A+B).C=A.C+D.C+C^2.
     \end{equation}
     If in this situation $A.B=4$, then  Proposition
     \ref{prop:splitting} shows $A.C=0$ and $A.D=A.B=4$, and therefore
     \eqref{eq:c:6} leads a contradiction,
     since the right hand side of Equation \eqref{eq:c:8} is
     $A.C+D.C+C^2=0+2-1=1$. We, therefore, conclude that $A.B= 3$, and
     as above we get from Lemma \ref{lem:AB}
     \begin{displaymath}
       2\geq B^2=2\cdot C.D+C^2= 2k\cdot C.E-1\geq 4\cdot C.E-1 \geq -1,
     \end{displaymath}
     which is only possible for $C.E=C.D=0$. But then \eqref{eq:c:8}
     implies $A.C\geq 3$, and since $A$ is big and $D$ has no fixed
     component  Lemma \ref{lem:big}
     gives the contradiction
     \begin{displaymath}
       1\leq A.D=A.B-A.C\leq 0.
     \end{displaymath}

     This finishes the cases where $\length(Z_p')=4$ and $D^2=0$.

     In Cases 5, 10 and 11 we apply Lemma \ref{lem:curves} to the
     irreducible curve $C$ with $C^2=0$ and find
     \begin{displaymath}
       (A+B).C\geq 3.
     \end{displaymath}
     In Cases 5 and 10 Equation \eqref{eq:c:4} then gives the
     contradiction
     \begin{displaymath}
       4\geq A.B \geq 3-C^2+D^2+1\geq 5,
     \end{displaymath}
     and similarly in Case 11 we get
     \begin{displaymath}
       3\geq A.B \geq 3-C^2+D^2+1=4.
     \end{displaymath}
     In very much the same way we get in Case 8
     \begin{displaymath}
       (A+B).C\geq 2
     \end{displaymath}
     and the contradiction
     \begin{displaymath}
       3\geq A.B\geq 2-C^2+D^2+1=4.
     \end{displaymath}

     Let us next have a look at the Cases 16 and 17. Consider the
     decomposition of the general $D=\sum_{i=1}^s E_i$ into irreducible
     components, none of which is fixed. In Case 16 we have $D^2=0$,
     and thus $E_i.E_j=0$ for all $i,j$, while in Case 17 we have
     $D^2=1$ and we thus may assume $E_1^2=1$ and $E_i.E_j=0$ for all
     $(i,j)\not=(1,1)$. Moreover, in both cases $C.D=0$ and thus
     $C.E_i=0$ for all $i$. Applying Lemma \ref{lem:curves} to $E_i$
     we find
     \begin{displaymath}
       A.E_i=(A+B).E_i-E_1.E_i\geq 3,
     \end{displaymath}
     and by \eqref{eq:c:2} we get
     \begin{equation}\label{eq:c:9}
       A.C=A.C+D.C\geq r-C^2\geq 0.
     \end{equation}
     But then
     \begin{displaymath}
       3\geq A.B=A.C+\sum_{i=1}^s A.E_i\geq 3s,
     \end{displaymath}
     which implies $s=1$ and $A.C=0$. From \eqref{eq:c:9} we deduce
     that $r=C^2=1$, and thus $C$ is irreducible with $C^2=1$.
     Similarly in Case 13 we have by \eqref{eq:c:2}
     \begin{displaymath}
       0=A.C+D.C\geq r-C^2=r-1\geq 0,
     \end{displaymath}
     and again $C$ is irreducible with $C^2=1$.
     Applying now Lemma
     \ref{lem:curves} to $C$ we get in each of these three cases the contradiction
     \begin{displaymath}
       4\leq (A+B).C=A.C+D.C+C^2=1.
     \end{displaymath}
     This outrules the Cases 13, 16, and 17.

     It remains to consider Case 9. Here we deduce from \eqref{eq:c:4}
     that
     \begin{displaymath}
       2\geq (A+B).C\geq r=2,
     \end{displaymath}
     and hence
     \begin{displaymath}
       2=(A+B).C=A.C+D.C+C^2=D.C.
     \end{displaymath}
     But then Lemma \ref{lem:AB} leads to the final contradiction
     \begin{displaymath}
       2=A.B-1\geq B^2=D^2+2\cdot D.C+C^2=4.
     \end{displaymath}
   \end{proof}

   It follows that $B_p=D_p$, $\kb=\kd$, and that $B_p$ is nef.

   \section{The General Case}\label{sec:generalcase}

   Let us review the situation and recall some notation. We are
   considering a divisor $L$ such $L$ and $L-K$ are very ample with
   $(L-K)^2>16$, and such that for a general point $p\in S$ the
   general element $L_p\in |L-3p|$ has no triple component through
   $p$ and that the equimultiplicity ideal of $L_p$ in $p$ in suitable
   local coordinates is one of the ideals in Table \eqref{eq:3jets} -- and
   for all $p$ the ideals have the same length. Moreover, we
   know that there is an algebraic family 
   $\kb=\overline{\{B_p\;|\;p\in S\}}\subset |B|_a$ without fixed
   component such that for a general point $p\in S$ 
   \begin{displaymath}
     B_p\in |\kj_{Z'_p/S}(L-K-A_p)|,
   \end{displaymath}
   where $Z'_p$ is the equimultiplicity scheme  of $L_p$ and
   $A_p$ is the unique divisor linearly equivalent to $L-K-B_p$ such that $B_p$ and $A_p$
   destabilize the vector bundle $\ke_p$ in \eqref{eq:vectorbundle}.
   Keeping these notations in mind we can now consider the two cases that
   either $\length(Z'_p)=4$ or $\length(Z'_p)=3$.

   \begin{proposition}\label{prop:length4}
     Let $p\in S$ be general and suppose that $\length(Z'_p)=4$. Then
     $B_p=E_p+F_p$, $E_p$ and $F_p$ are irreducible, smooth, elliptic
     curves, $E_p^2=F_p^2=0$, $E_p.F_p=1$,
     $A.E_p=A.F_p=2$, $L.E_p=L.F_p=3$,  $A.B=4$, $K.E_p=K.F_p=0$,
     and
     $\exists\;s\in H^0\big(B_p,\ko_{B_p}(A_p)\big)$ such that $Z_p'=\{s=0\}$. 


     Moreover, neither $|E|_a$ and $|F|_a$ is a linear system, but
     they both induce an elliptic fibration with section on $S$ over an elliptic
     curve.
   \end{proposition}
   \begin{proof}
     Since $A^2>0$ we can apply the Hodge Index Theorem (see e.g.\
     \cite{BHPV04}), and since $(A+B)^2\geq 17$ by assumption and
     $A.B\leq 4$ by Equation \eqref{eq:AB:2} we
     deduce 
     \begin{multline}\label{eq:length4:0}
       16\geq (A.B)^2\geq A^2\cdot B^2
       =\big((A+B)^2-2A.B-B^2\big).B^2\\\geq (9-B^2)\cdot B^2.
     \end{multline}
     In Section \ref{sec:zero} we have shown that $B=D$ is nef, and thus
     Lemma \ref{lem:AB} together with Equation \eqref{eq:length4:0}
     shows 
     \begin{equation}\label{eq:length4:1}
       0\leq B^2\leq 2.
     \end{equation}
     Then, however, Lemma \ref{lem:DD} implies that $B_p$ must be
     reducible. 

     Let us first consider the case that the part of $B_p$ through $p$
     is reduced.  Then
     by Lemma \ref{lem:DD},   Lemma \ref{lem:AB}, and Equations
     \eqref{eq:AB:2} and \eqref{eq:length4:1} we know that
     $B_p=E_p+F_p+R$, where $E_p$ and 
     $F_p$ are irreducible and smooth in $p$. In particular,
     $E_p.F_p\geq 1$, and thus
     \begin{multline*}
       2\geq B^2=E_p^2+2\cdot E_p.F_p+F_p^2+2\cdot(E_p+F_p).R+R^2\\
       \geq 2+2\cdot(E_p+F_p).R.
     \end{multline*}
     Since $E_p.F_p=1$ and since the components 
     $E_p$ and $F_p$  vary in at
     least one-dimensional families and $R$ has no fixed component, $(E_p+F_p).R\geq 1$, unless
     $R=0$. This would however give a contradiction, so
     $R=0$. Therefore necessarily, $B_p=E_p+F_p$, $E_p.F_p=1$, and $E_p^2=F_p^2=0$.
     Then by Lemma \ref{lem:curves} $(A+B).E_p\geq 3$ and
     $(A+B).F_p\geq 3$, so that
     \begin{displaymath}
       4\geq A.B\geq (A+B).E_p+(A+B).F_p-B^2\geq 4
     \end{displaymath}
     implies $E_p.A_p=2=F_p.A_p$ and $(A+B).E_p=3=(A+B).F_p$. Applying
     Lemma \ref{lem:curves} once more, we see that 
     \begin{equation}
       \label{eq:length4:2}
       p_a(E_p)\leq 1\;\;\;\mbox{ and }\;\;\;p_a(F_p)\leq 1.
     \end{equation}

     We claim that in $p$ the curve $L_p$ can share at most with one
     of $E_p$ or $F_p$ a common tangent, and it can do so at most with
     multiplicity one. For this consider local coordinates $(x_p,y_p)$ as in the Table 
     \eqref{eq:3jets}. Since $\length(Z_p')=4$ we know that
     $\kj_{Z_p',p}=\langle x_p^2,y_p^2\rangle$ does not contain $x_py_p$, and since
     $B_p=E_p+F_p\in|\kj_{Z_p'}(L-K-A)|$, where $E_p$ and $F_p$ are
     smooth in $p$, we deduce that in local coordinates their
     equations are
     \begin{displaymath}
       x_p+a\cdot y_p+h.o.t.\;\;\;\mbox{ respectively }\;\;\;x_p-a\cdot y_p+h.o.t.,
     \end{displaymath}
     where $a\not=0$. By Table \eqref{eq:3jets} the local equation
     $f_p$ of $L_p$ has either $\jet_3(f_p)=x_p^3$ and has thus no
     common tangent with either $E_p$ or $F_p$, or
     $\jet_3(f_p)=x_p^3-y_p^3$ and it is divisible at most once by one
     of $x_p-ay_p$ or $x_p+ay_p$ .

     In particular, $E_p$ can at most once be a component of $L_p$,
     and we deduce
     \begin{displaymath}
       E_p.K_S=E_p.L_p-E_p.A_p-E_p.B_p=E_p.L_p-3\geq
       \left\{
         \begin{array}{ll}
           0,&\mbox{ if } E_p\not\subset L_p,\\
           -1, &\mbox{ if } E_p\subset L_p.
         \end{array}
       \right.
     \end{displaymath}
     But then, since the genus is an integer,
     \begin{displaymath}
       p_a(E_p)=\frac{E_p^2+E_p.K_S}{2}+1=\frac{E_p.K_S}{2}+1\geq 1,
     \end{displaymath}
     in which case \eqref{eq:length4:2} gives $p_a(E_p)=1$. This
     shows, in particular, that 
     \begin{displaymath}
       K.E_p=0\;\;\;\mbox{ and }\;\;\;L_p.E_p=3.
     \end{displaymath}

     By symmetry the same holds for $F_p$.

     Since $E_p^2=0$ the family $|E|_a$ is a pencil and induces an
     elliptic fibration on $S$ (see \cite{Kei01} App.\ B.1). In
     particular, the generic element $E_p$ in $|E|_a$ must be
     smooth (see e.g.\ \cite{BHPV04} p.\ 110). And with the
     same argument the generic element $F_p$ in $|F|_a$ is smooth.

     Suppose now that $|E|_a$ is a linear system. 
     Since $E.F=1$ and for $q\in F_p$ general $E_q\cap F_p=\{q\}$ the
     linear system $|\ko_{F_p}(E)|$  
     is a $\mathfrak{g}_1^1$ on the smooth curve $F_p$ implying that 
     $F_p$ is rational contradicting $p_a(F_p)=1$. Thus $|E|_a$ is not
     linear, and analogously $|F|_a$ is not.


     It remains to consider the case that $B_p$ is not reduced in
     $p$. Using the notation of the proof of Lemma \ref{lem:DD} we write $B_p\equiv
     k\cdot E_p+E'$ with $k\geq 2$, $E_p$ irreducible passing through $p$ and
     $E'$ not containing any component algebraically equivalent to $E_p$. We
     have seen there (see p.\ \pageref{eq:DD:0}) that $E'\not=0$ implies $B_p^2\geq 4$ in
     contradiction to Lemma \ref{lem:AB}. We may therefore assume
     $B_p=k\cdot E_p$ with $E_p^2\geq 0$. If $E_p^2\geq 1$, then again
     $B_p^2\geq 4$. Thus $E_p^2=0$. Applying Lemma \ref{lem:curves} to
     $E_p$ we get
     \begin{displaymath}
       3\leq (A+B).E_p=A.E_p,
     \end{displaymath}
     and hence the contradiction
     \begin{displaymath}
       4\geq A.B=k\cdot A.E_p\geq 6.
     \end{displaymath}
     Therefore, $B_p$ must be reduced in $p$.
   \end{proof}

   \begin{proposition}\label{prop:length3}
     Let $p\in S$ be general and suppose that
     $\length(Z'_p)=3$.  
     Then $B_p$ is an irreducible, smooth, rational curve
     in the pencil $|B|_a$ with $B^2=0$, $A.B=3$ and
     $\exists\;s\in H^0\big(B_p,\ko_{B_p}(A_p)\big)$ such that $Z_p'$
     is the zero-locus of $s$.

     In particular, $S\rightarrow |B|_a$
     is a ruled surface and $2B_p$ is a 
     fixed component of $|L-3p|$.
   \end{proposition}
   \begin{proof}
     Since $A^2>0$ we can apply the Hodge Index Theorem (see e.g.\
     \cite{BHPV04}), and since $(A+B)^2\geq 17$ by assumption and
     $A.B\leq 3$ by Equation \eqref{eq:AB:2} we
     deduce 
     \begin{displaymath}
       9\geq (A.B)^2\geq A^2\cdot B^2
       =\big((A+B)^2-2A.B-B^2\big)\geq (11-B^2)\cdot B^2.
     \end{displaymath}
     Since in Section \ref{sec:zero} we have shown that $B$ is nef,
     this inequality together with
     Lemma \ref{lem:AB} implies
     \begin{equation}\label{eq:length3:0}
       B^2=0.
     \end{equation}
     Once we have shown that $B_p$ is irreducible and reduced, we then know that
     $|B|_a$ is a pencil and induces a fibration on $S$ whose fibres
     are the elements of $|B|_a$ (see \cite{Kei01} App.\ B.1). In
     particular, the general element of $|B|_a$, which is $B_p$, is
     smooth (see \cite{BHPV04} p.\ 110).
     
     Let us therefore first show that $B_p$ is irreducible and reduced. Since $\kb$
     has no fixed component we know for each irreducible component
     $B_i$ of $B_p=\sum_{i=1}^rB_i$ that $B_i^2\geq 0$, and hence by Lemma
     \ref{lem:curves} that $(A+B).B_i\geq
     3$. Thus by \eqref{eq:AB:2} and \eqref{eq:length3:0}
     \begin{displaymath}
       3\cdot r\leq (A+B).B=A.B+B^2=A.B\leq 3,
     \end{displaymath}
     which shows that $B_p$ is irreducible and reduced and that $A.B=3$.
     Moreover, $(A+B).B=3$, and Lemma \ref{lem:curves} implies that
     \begin{equation}
       \label{eq:length3:1}
       p_a(B_p)\leq 1.
     \end{equation}
     
     Since $A.B=3=\length(Z'_p)$ Proposition \ref{prop:splitting}
     implies that there is a section
     $s_p\in H^0\big(B_p,\ko_{B_p}(A_p)\big)$ such that
     $Z_p'$ is the zero-locus of $s_p$, which is just $3p$.
     Note that for $p\in S$ general and $q\in B_p$ general we have
     $B_p=B_q$ since $|B|_a$ is a pencil, and thus by the construction
     of $B_p$ and $B_q$ we also have
     \begin{displaymath}
       A_p\sim_l L-K-B_p=L-K-B_q\sim_l A_q.
     \end{displaymath}
     But if $A_p$ and $A_q$ are linearly equivalent, then so are the
     divisors $s_p$ and $s_q$ induced on the curve $B_p=B_q$. The
     curve $B_p$ therefore contains a linear series $|\ko_{B_p}(A_p)|$
     of degree three which contains $3q$ for a general point $q\in
     B_p$. If $B_p$ was an elliptic curve, then $|\ko_{B_p}(A_p)|$
     would necessarily have to be a $\mathfrak{g}_3^2$ embedding $B_p$
     as a plane curve of degree three and the general point $q$ would
     be an inflexion point. But that is clearly not possible. Thus
     \begin{displaymath}
       p_a(B_p)=0,
     \end{displaymath}
     and by the adjunction formula we get
     \begin{equation}
       \label{eq:length3:2}
       K.B=2p_a(B)-2-B^2=-2.
     \end{equation}

     Note also, that $Z'_p\subset B_p$ in view of Table \eqref{eq:3jets}
     implies that $B_p$ and $L_p$ have a common tangent in $p$.
     Suppose that $B_p$ and $L_p$
     have no common component, i.\ e.\ $B_p\not\subset L_p$, then
     \begin{displaymath}
       3\leq\mult_p(B_p)\cdot\mult_p(L_p)< L.B=A.B+B^2+K.B=3+K.B=1,
     \end{displaymath}
     which contradicts \eqref{eq:length3:2}. 
     Thus, $B_p$ is at least once contained in $L_p$. 
     Moreover, if
     $2B_p\not\subset L_p$ then by Table \eqref{eq:3jets} $L'_p:=L_p-B_p$
     has multiplicity two in $p$, and it still has a common tangent
     with $B_p$ in $p$, so that
     \begin{equation}\label{eq:length3:3}
       3\leq L'_p.B_p=L.B-B^2=A.B+K.B=3+K.B=1
     \end{equation}
     again is impossible. We conclude finally, that $B_p$ is at least
     twice contained in $L_p$

     Note finally,
     since $\dim|B|_a=1$ there is a unique curve $B_p$ in $|B|_a$
     which passes through $p$, i.\ e.\ it does not depend on the
     choice of $L_p$, so that in these cases $B_p$ respectively $2B_p$
     is actually a fixed component of $|L-3p|$.
   \end{proof}


   \section{Regular Surfaces}\label{sec:regular}

   \begin{theorem}[``If $S$ is regular, then $S$ is a rationally ruled surface.'']\label{thm:regular}
     More precisely, let $S$ be a regular surface and $L$ a line bundle
     on $S$ such that $L$ and $L-K$ are very ample. Suppose that
     $(L-K)^2>16$ and that for a
     general $p\in S$ the linear system $|L-3p|$ contains a curve
     $L_p$ which has no triple component through $p$, but such that
     $h^1\big(\kj_{Z_p}(L)\big)\not=0$ where $Z_p$ is the
     equimultiplicity scheme of $L_p$ at $p$.

     Then there is a
     rational ruling $\pi:S\rightarrow \PC^1$ of $S$ such that $L_p$ contains
     the fibre over $\pi(p)$ with multiplicity two.
   \end{theorem}
   \begin{proof}
     Let us suppose that $S$ is regular, so that each algebraic family
     is indeed a linear system, and let $p\in S$ be general.

     The case $\length(Z'_p)=4$ is excluded since by
     Proposition \ref{prop:length4} the algebraic families $|E|_a$
     and $|F|_a$ would have to be linear systems.
     Thus necessarily  $\length(Z'_p)=3$, and we are done by
     Proposition \ref{prop:length3}.
   \end{proof}

   \bibliographystyle{amsalpha-tom}

   \providecommand{\bysame}{\leavevmode\hbox to3em{\hrulefill}\thinspace}

\end{document}